\documentclass[preprint,11pt]{elsarticle}

\usepackage{amsfonts, amsmath, amscd}
\usepackage[psamsfonts]{amssymb}

\usepackage{amssymb}

\usepackage{pb-diagram}

\usepackage[all,cmtip]{xy}

\usepackage[usenames]{color}

\newtheorem{theorem}{Theorem}[section]

\newtheorem{corollary}[theorem]{Corollary}

\newtheorem{proposition}[theorem]{Proposition}

\newproof{proof}{Proof}

\numberwithin{equation}{section}

\newcommand{\CC}{C_k}
\newcommand{\NN}{\mathbb{N}}

\newcommand{\w}{\omega}

\newcommand{\KK}{\mathcal{K}}

\newcommand{\IR}{\mathbb{R}}

\newcommand{\e}{\varepsilon}
\newcommand{\cl}{\mathrm{cl}}

\renewcommand{\phi}{\varphi}

\newcommand{\SI}{\underrightarrow{\mbox{\rm s-ind}}_n\,}

\input xy
\xyoption{all}

\begin{document}

\begin{frontmatter}

\title{The $\kappa$-Fr\'{e}chet--Urysohn property for locally convex spaces}

\author{S.~Gabriyelyan}%\tnotetext[label1]{The first named author was partially supported by Israel Science Foundation grant 1/12.}
\ead{saak@math.bgu.ac.il}
\address{Department of Mathematics, Ben-Gurion University of the Negev, Beer-Sheva, P.O. 653, Israel}

\begin{abstract}
A topological space $X$ is $\kappa$-Fr\'{e}chet--Urysohn if for every open subset $U$ of $X$ and every $x\in \overline{U}$ there exists a sequence in $ U$ converging to $x$.
We prove that every $\kappa$-Fr\'{e}chet--Urysohn Tychonoff space $X$ is Ascoli. We apply this statement and some of known results to characterize the $\kappa$-Fr\'echet--Urysohn property in various important classes of locally convex spaces. In particular, answering a question posed in \cite{GGKZ} we obtain that $C_p(X)$ is Ascoli iff $X$ has the property $(\kappa)$.
%For a Tychonoff space  $X$,  denote by $C_p(X)$ and $\CC(X)$ the space $C(X)$ of all real-valued continuous functions on $X$ endowed with the pointwise topology and the compact-open topology, respectively. The space $X$ is Ascoli if every compact subset of $\CC(X)$ is evenly continuous. We prove that every $\kappa$-Fr\'{e}chet--Urysohn space $X$ is Ascoli. We apply this statement and some of the main results from \cite{Banakh-Survey,Gabr-LCS-Ascoli,Gab-LF,Gabr-L(X)-Ascoli,GGKZ,GKP,Sak2} to characterize the $\kappa$-Fr\'echet--Urysohn property in various important classes of locally convex spaces. In particular, answering a question posed in \cite{GGKZ} we show that $C_p(X)$ is Ascoli iff $X$ has the property $(\kappa)$.
\end{abstract}


\end{frontmatter}

%%%%%%%%%%%%%%%%%%%%%%%%%%%
%%%%%%%%%%%%%%%%%%%%%%%%%%%
%%%%%%%%%%%%%%%%%%%%%%%%%%%
%%%%%%%%%%%%%%%%%%%%%%%%%%%

%%%%%%%%%%%%%%%%%%%%%%%%%%%
%%%%%%%%%%%%%%%%%%%%%%%%%%%
%%%%%%%%%%%%%%%%%%%%%%%%%%%
%%%%%%%%%%%%%%%%%%%%%%%%%%%

\section{Introduction}

%%%%%%%%%%%%%%%%%%%%%%%%%%%
%%%%%%%%%%%%%%%%%%%%%%%%%%%
%%%%%%%%%%%%%%%%%%%%%%%%%%%
%%%%%%%%%%%%%%%%%%%%%%%%%%%

Following Arhangel'skii, a topological space $X$ is said to be {\em $\kappa$-Fr\'{e}chet--Urysohn} if for every open subset $U$ of $X$ and every $x\in \overline{U}$, there exists a sequence $\{ x_n\}_{n\in\NN} \subseteq U$ converging to $x$. Clearly, every Fr\'{e}chet--Urysohn space is $\kappa$-Fr\'{e}chet--Urysohn.  In \cite[Theorem~3.3]{LiL} Liu and Ludwig showed that a topological space $X$ is $\kappa$-Fr\'{e}chet--Urysohn if and only if $X$ is a $\kappa$-pseudo open image of a metric space. Below we give another characterization of $\kappa$-Fr\'{e}chet--Urysohn spaces, see \ref{t:Char-kFU}. It is known that there are $\kappa$-Fr\'{e}chet--Urysohn spaces which are not $k$-spaces, and there are sequential spaces which are not $\kappa$-Fr\'{e}chet--Urysohn, see \cite{LiL} or Proposition \ref{p:kFU-direct-sum} below.

Let $X$ be a Tychonoff (=completely refular and Hausdorff) space. Denote by $\CC(X)$ and $C_p(X)$ the space $C(X)$ of all real-valued continuous functions on $X$ endowed with the compact-open topology and the pointwise topology, respectively.  Following \cite{BG}, $X$ is called an {\em Ascoli space} if every compact subset $\KK$ of $\CC(X)$ is evenly continuous (i.e., if the map $(f,x)\mapsto f(x)$ is continuous as a map from $\KK\times X$ to $\IR$). In \cite{Gabr-LCS-Ascoli}  we noticed that $X$ is Ascoli if and only if every compact subset of $\CC(X)$ is equicontinuous.  The classical Ascoli theorem \cite[Theorem~3.4.20]{Eng} states that every $k$-space is Ascoli. %Moreover, every $k_\IR$-space is Ascoli by \cite{Noble}.

In \cite[Theorem 2.1]{Sak2}, Sakai characterized those spaces $C_p(X)$ which are $\kappa$-Fr\'{e}chet--Urysohn. Recall that a family $\{ A_i\}_{i\in I}$ of subsets of a set $X$ is said to be {\em  point-finite} if the set $\{i\in I: x\in A_i\}$ is finite for every $x\in X$. A family $\{ A_i\}_{i\in I}$ of subsets of a topological space $X$ is called {\em strongly point-finite} if for every $i\in I$, there exists an open set $U_i$ of $X$ such that $A_i\subseteq U_i$ and  $\{ U_i\}_{i\in I}$ is point-finite.  Following Sakai \cite{Sak2}, a topological space $X$ is said to have the {\em property $(\kappa)$} if every sequence of pairwise disjoint finite subsets  of $X$ has a strongly point-finite subsequence.
\begin{theorem}[\cite{Sak2}] \label{t:Sakai-Cp-k-FU}
The space $C_p(X)$ is $\kappa$-Fr\'{e}chet--Urysohn if and only if $X$ has the property $(\kappa)$.
\end{theorem}
A characterization of the spaces $\CC(X)$ which are $\kappa$-Fr\'{e}chet--Urysohn is given in \cite{Sakai-3}.

In \cite{GGKZ} we proved the following theorem.
\begin{theorem}[\cite{GGKZ}] \label{t:Cp-Ascoli-k-FU}
If $C_p(X)$ is Ascoli,  then it is $\kappa$-Fr\'echet--Urysohn.
\end{theorem}

However, the question (see \cite[Question~2.4]{GGKZ}) of whether every $\kappa$-Fr\'echet--Urysohn space $C_p(X)$ is Ascoli remained open. In this short note we answer this question in the affirmative using the following somewhat unexpected result.

\begin{theorem} \label{t:k-FU-seq-Ascoli}
Each $\kappa$-Fr\'{e}chet--Urysohn space $X$ is Ascoli.
\end{theorem}

Now Theorems \ref{t:Sakai-Cp-k-FU}-\ref{t:k-FU-seq-Ascoli} immediately imply the following characterization of spaces $C_p(X)$ which are Ascoli.
\begin{corollary} \label{t:Cp-seq-Ascoli}
Let $X$ be a Tychonoff space. Then $C_p(X)$ is Ascoli if and only if $X$ has the property $(\kappa)$.
\end{corollary}

Denote by $\mathcal{D}(\Omega)$  the space  of test functions over an open subset $\Omega$ of $\IR^n$.  In \cite{Gab-LF} we proved that $\mathcal{D}(\Omega)$ and the strong dual $\mathcal{D}'(\Omega)$ of $\mathcal{D}(\Omega)$, the space of distributions, are not Ascoli. Therefore, by Theorem \ref{t:k-FU-seq-Ascoli}, $\mathcal{D}(\Omega)$ and $\mathcal{D}'(\Omega)$ are not  $\kappa$-Fr\'{e}chet--Urysohn spaces.
Below we apply Theorem \ref{t:k-FU-seq-Ascoli} and some of the main results from \cite{Banakh-Survey,Gabr-LCS-Ascoli,Gab-LF,Gabr-L(X)-Ascoli,GKP} to characterize the $\kappa$-Fr\'echet--Urysohness in various important classes of locally convex spaces.

%%%%%%%%%%%%%%%%%%%%%%%%%%%%%%%%
%%%%%%%%%%%%%%%%%%%%%%%%%%%%%%%%
%%%%%%%%%%%%%%%%%%%%%%%%%%%%%%%%
%%%%%%%%%%%%%%%%%%%%%%%%%%%%%%%%

\section{Proof of Theorem \ref{t:k-FU-seq-Ascoli} }

%%%%%%%%%%%%%%%%%%%%%%%%%%%%%%%%
%%%%%%%%%%%%%%%%%%%%%%%%%%%%%%%%
%%%%%%%%%%%%%%%%%%%%%%%%%%%%%%%%
%%%%%%%%%%%%%%%%%%%%%%%%%%%%%%%%

We start from the following characterization of $\kappa$-Fr\'{e}chet--Urysohn spaces. The closure of a subset $A$ of a topological space $X$ is denoted by $\overline{A}$ or $\cl_X(A)$.

\begin{theorem} \label{t:Char-kFU}
A topological space $X$ is $\kappa$-Fr\'{e}chet--Urysohn if and only if each point $x\in X$ is contained in a dense $\kappa$-Fr\'{e}chet--Urysohn subspace of $X$.
\end{theorem}

\begin{proof}
The necessity is clear. To prove sufficiency, fix an open subset $U$ of $X$ and a point $x\in \overline{U}$. Let $Y$ be a dense $\kappa$-Fr\'{e}chet--Urysohn subspace of $X$ containing $x$. Then $V:= U\cap Y$ is an open subset of $Y$. 
We claim that $x\in \cl_Y(V)$. Indeed, if $W\subseteq Y$ is an open neighborhood of $x$ in $Y$, take an open $W'\subseteq X$ such that $W=W'\cap Y$. Then the set $W' \cap U$ is open in $X$. Since $Y$ is dense in $X$ the set
$
(W'\cap U)\cap Y = (W'\cap Y) \cap (U\cap Y) = W\cap V
$
is not empty. Thus $x\in \cl_Y(V)$ and the claim is proved.
Finally, since $Y$ is $\kappa$-Fr\'{e}chet--Urysohn there is a sequence $\{ y_n\}_{n\in\NN} \subseteq V\subseteq U$ converging to $x$. \qed
\end{proof}

\begin{corollary} \label{c:kFU-dense-subgroup}
Let $Y$ be a dense subset of a homogeneous space (in particular, a topological group) $X$. If $Y$ is $\kappa$-Fr\'{e}chet--Urysohn, then $X$ is also a $\kappa$-Fr\'{e}chet--Urysohn.
\end{corollary}
\begin{proof}
Fix arbitrarily $y_0\in Y$. Let $x\in X$. Take a homeomorphism $h$ of $X$ such that $h(y_0)=x$. Then $x\in h(Y)$ and $h(Y)$ is a $\kappa$-Fr\'{e}chet--Urysohn space. Therefore, each element of $X$ is contained in a dense $\kappa$-Fr\'{e}chet--Urysohn subspace of $X$ and Theorem \ref{t:Char-kFU} applies. \qed
\end{proof}

In \cite[Theorem~4.1]{LiL} Liu and Ludwig proved that the product of a family of bi-sequential spaces is $\kappa$-Fr\'{e}chet--Urysohn. Note that any countable product of bi-sequential spaces is bi-sequential, see \cite[Proposition~3.D.3]{Michael-quotient}. On the other hand, countable products of $W$-spaces are $W$-spaces (\cite[Theorem~4.1]{Gruenhage-games}) and there are $W$-spaces which are not bi-sequential (\cite[Example~5.1]{Gruenhage-games}). Taking into account that bi-sequential spaces and $W$-spaces are Fr\'{e}chet--Urysohn spaces, the next corollary essentially generalizes Theorem 4.1 of \cite{LiL}.

\begin{corollary} \label{c:kFU-product}
Let $\{X_i:i\in I\}$ be a family of topological spaces such that $\prod_{i\in I'}X_i$ is Fr\'echet--Urysohn for any countable subset $I'$ of $I$. Then the space $X=\prod_{i\in I}X_i$ is $\kappa$-Fr\'{e}chet--Urysohn.
\end{corollary}

\begin{proof}
For every $z=(z_i)\in X$, set
\[
\sigma(z) := \big\{ x=(x_i)\in X: \{ i: x_i\not= z_i\} \mbox{ is finite}\big\}.
\]
Clearly, $\sigma(z)$ is a dense subspace of $X$. Proposition %\ref{p:Noble-Sigma-product}
2.6 of \cite{GGKZ} states that $\sigma(z)$ is Fr\'echet--Urysohn. By  Theorem \ref{t:Char-kFU}, $X$ is $\kappa$-Fr\'{e}chet--Urysohn. \qed
\end{proof}

Below we prove Theorem \ref{t:k-FU-seq-Ascoli}.

\smallskip
%\begin{proof}
{\em Proof of Theorem \ref{t:k-FU-seq-Ascoli}.}
Suppose for a contradiction that $X$ is not an Ascoli space. Then there exists a compact set $K$ in $\CC(X)$ which is not equicontinuous at some point $z\in X$. Therefore there is $\e_0 >0$ such that for every open neighborhood $U$ of $z$ there exists a function $f_U\in K$ for which the open set $W_{f_U} :=\{ x\in U: | f_U (x)- f_U(z)|> \e_0\}$ is not empty. Set
\[
W:= \bigcup\{ W_{f_U}: U \mbox{ is an open neighborhood of } z\}.
\]
Then $W$ is an open subset of $X$ such that $z\in\overline{W}\setminus W$. As $X$ is $\kappa$-Fr\'{e}chet--Urysohn, there is a sequence $\{ x_n :n\in\NN\} \subseteq W$ converging to $z$. For every $n\in\NN$, choose an open neighborhood $U_n$ of $z$ such that $x_n\in W_{f_{U_n}}  (\subseteq U_n)$ and, therefore,
\begin{equation} \label{equ:kFU-Ascoli-1}
| f_{U_n} (x_n)- f_{U_n}(z)|> \e_0 \quad (\mbox{for all } n\in\NN).
\end{equation}
Set $S:= \{ x_{n}: n\in\NN\} \cup \{ z\}$. Then $S$ is a compact subset of $X$. Denote by $p$ the restriction map $p: \CC(X)\to \CC(S), p(f)=f|_S$. Then $p(K)$ is a compact subset of the Banach space $\CC(S)$. Applying the Ascoli theorem to the compact space $S$ we obtain that the sequence $\{ p(f_{U_n})\}_{n\in\NN}\subseteq p(K)$ is equicontinuous at $z\in S$ and, therefore, there is an $N\in\NN$ such that
\[
\big| f_{U_n} (x_{i}) - f_{U_n}(z) \big| <\frac{\e_0}{2} \; \; \mbox{ for all } \, i\geq N \, \mbox{ and } \, n\in\NN.
\]
In particular, for $i=n=N$ we obtain
$
\big| f_{U_N} (x_{N}) - f_{U_N}(z) \big| <\frac{\e_0}{2}.
$
But this contradicts (\ref{equ:kFU-Ascoli-1}). Thus $X$ is an Ascoli space. \qed
%\end{proof}

\smallskip
The next corollary strengthens Theorem 1.3 of \cite{GGKZ}.
\begin{corollary} \label{t:Cech-complete-Cp-Ascoli}
Let $X$ be a \v{C}ech-complete space. Then $C_p(X)$ is Ascoli if and only if $X$ is scattered.
\end{corollary}

\begin{proof}
If $C_p(X)$ is Ascoli, then $X$ is scattered by Theorem 1.3 of \cite{GGKZ}. Conversely, if $X$ is scattered, then, by Corollary 3.8 of \cite{Sak2}, $X$ has the property $(\kappa)$. Thus, by Corollary \ref{t:Cp-seq-Ascoli}, $C_p(X)$ is Ascoli.\qed
\end{proof}

Let $E$ be a locally convex space over a field $\mathbf{F}$, where $\mathbf{F}=\IR$ or $\mathbb{C}$, and let $E'$ the dual space of $E$. If $E$ is a Banach space, denote by $B$ the closed unit ball of $E$ and set
$
B_w:=\big(B,\sigma(E,E')|_B\big), %E_w:=\big(E,\sigma(E,E')\big) \mbox{ and } \;
$
where $\sigma(E,E')$ is the weak topology on $E$.

\begin{corollary} \label{c:Banach-kappa-FU}
{\rm (i)} If $E$ is a Banach space, then $B_w$ is $\kappa$-Fr\'{e}chet--Urysohn if and only if $E$ does not contain an isomorphic copy of $\ell_1$.

{\rm (ii)} A Fr\'{e}chet space $E$ over $\mathbf{F}$ is $\kappa$-Fr\'{e}chet--Urysohn in the weak topology if and only if $E=\mathbf{F}^N$ for some $N\leq\w$.

{\rm (iii)} If $X$ is a $\mu$-space and a $k_\IR$-space, then $\CC(X)$ is $\kappa$-Fr\'{e}chet--Urysohn in the weak topology if and only if $X$ is discrete.

{\rm (iv)} The weak${}^\ast$ dual space of a metrizable barrelled space $E$ is $\kappa$-Fr\'{e}chet--Urysohn if and only if $E$ is finite-dimensional.
\end{corollary}

\begin{proof}
(i) Theorem 1.9 of \cite{GKP} or Theorem 6.1.1 and Corollary 1.7 of \cite{Banakh-Survey} state that $B_w$ is Ascoli if and only if $B_w$ is Fr\'{e}chet--Urysohn if and only if $E$ does not contain an isomorphic copy of $\ell_1$. Now Theorem \ref{t:k-FU-seq-Ascoli} applies.

(ii) Corollary 1.7 of \cite{Gabr-LCS-Ascoli} states that $E$ is Ascoli in the weak topology if and only if $E=\mathbf{F}^N$ for some $N\leq\w$. This result and Theorem \ref{t:k-FU-seq-Ascoli} imply the desired.

(iii) Corollary 1.9 of \cite{Gabr-LCS-Ascoli} states that $\CC(X)$ is Ascoli in the weak topology if and only if $X$ is discrete. Now the assertion follows from Theorem \ref{t:k-FU-seq-Ascoli} and the fact that any product of metrizable spaces is $\kappa$-Fr\'{e}chet--Urysohn (see Fact 1.2 of \cite{LiL}).

(iv) Corollary 1.14 of \cite{Gabr-LCS-Ascoli} states that the weak${}^\ast$ dual space of $E$ is Ascoli if and only if $E$ is finite-dimensional, and Theorem \ref{t:k-FU-seq-Ascoli} applies. \qed
\end{proof}

Now we consider direct locally convex sums of locally convex spaces. The simplest infinite direct sum of lcs is the space $\phi$, the direct locally convex sum $\bigoplus_{n\in\NN} E_n$ with $E_n=\mathbf{F}$ for all $n\in\NN$. It is well known that $\phi$ is a sequential non-Fr\'{e}chet--Urysohn space, see Example 1 of \cite{nyikos}.

\begin{proposition} \label{p:kFU-direct-sum}
An infinite direct sum of (non-trivial) locally convex spaces is not $\kappa$-Fr\'{e}chet--Urysohn. In particular, $\phi$ is not a $\kappa$-Fr\'{e}chet--Urysohn space.
\end{proposition}

\begin{proof}
Let $L=\bigoplus_{i\in I} E_i$ be the direct locally convex sum of an infinite family $\{ E_i\}_{i\in I}$ of locally convex spaces. It is well known %(see Proposition \ref{p:Finite-direct-summund})
that every $E_i$ can be represented as a direct sum $\mathbf{F}\oplus E'_i$. Therefore $L$ contains $\phi$ as a direct summand. Since the projection of $L$ onto $\phi$ is open and the $\kappa$-Fr\'{e}chet--Urysohn property is preserved under open maps (see Proposition 3.3 of \cite{LiL}), it is sufficient to show that $\phi$ is not a $\kappa$-Fr\'{e}chet--Urysohn space.

We consider elements of $\phi$ as functions from $\NN$ to $\mathbf{F}$ with finite support. Recall that the sets of the form
\begin{equation} \label{equ:kFU-Ascoli-2}
\{ f\in \phi: |f(n)|<\e_n \mbox{ for every } n\in\NN\},
\end{equation}
where $\e_n >0$ for all $n\in\NN$, form a basis at $0$ of $\phi$ (see for example \cite[Example~1]{nyikos}).
For every $n,k\in\NN$, set
\[
U_{n,k}:= \left\{ f\in \phi: |f(1)|>\frac{1}{2n} \mbox{ and } |f(n)|>\frac{1}{2k} \right\},
\]
and set $U:= \bigcup_{n,k\in\NN} U_{n,k}$. It is easy to see that all the sets $U_{n,k}$ are open in $\phi$ and $0\not\in U_{n,k}$. Hence $U$ is an open subset of $\phi$ such that $0\not\in U$. To show that $\phi$ is not $\kappa$-Fr\'{e}chet--Urysohn, it suffices to prove that (A) $0\in \overline{U}$, and (B) there is no a sequence in $U$ converging to $0$.

(A) Let $W$ be a basic neighborhood of zero in $\phi$ of the form (\ref{equ:kFU-Ascoli-2}). Choose an $n\in\NN$ such that $\frac{1}{n}<\e_1$, and take $k\in\NN$ such that $\frac{1}{k}<\e_n$. It is clear that $U_{n,k}\cap W$ is not empty. Thus $0\in \overline{U}$.

(B) Suppose for a contradiction that there is a sequence $S=\{ f_j\}_{j\in\NN}$ in $U$ converging to $0$. For every $j\in\NN$, take $n_j,k_j\in\NN$ such that $f_j \in U_{n_j,k_j}$. Since $f_j \to 0$, the definition of $U_{n,k}$ implies that $\frac{1}{2n_j}<|f_j(1)|\to 0$, and hence $n_j \to \infty$. Without loss of generality we can assume that $1<n_1 <n_2<\cdots$. For every $n\in\NN$, define $\e_n =\frac{1}{4k_j}$ if $n=n_j$ for some $j\in\NN$, and $\e_n =1$ otherwise. Set
\[
V:=\{f\in \phi: |f(n)|<\e_n  \mbox{ for every } n\in\NN \}.
\]
Then, $V$ is a neighborhood of $0$, and the construction of $U_{n,k}$ implies that $V\cap U_{n_j,k_j} =\emptyset$ for every $j\in\NN$. Thus $S\cap V=\emptyset$ and hence $f_j \not\to 0$, a contradiction. \qed
\end{proof}

Recall that a {\em strict $(LF)$-space} $E$ is the direct limit $E=\SI E_n$ of an increasing sequence
\[
E_0 \hookrightarrow E_1 \hookrightarrow E_2 \hookrightarrow \cdots
\]
of Fr\'{e}chet (= locally convex complete metric linear) spaces in the category of locally convex spaces and continuous linear maps. The space $\mathcal{D}(\Omega)$ of test functions is one of the most famous and important examples of  strict $(LF)$-spaces.

\begin{corollary} \label{t:kFU-strict-LF}
A strict $(LF)$-space $E$ is $\kappa$-Fr\'{e}chet--Urysohn if and only if $E$ is a Fr\'{e}chet space.
\end{corollary}

\begin{proof}
Theorem 1.2 %\ref{t:Ascoli-strict-LF}
of \cite{Gab-LF} states that $E$ is an Ascoli space  if and only if $E$ is a Fr\'{e}chet space or $E=\phi$. Now the assertion follows from Theorem \ref{t:k-FU-seq-Ascoli} and Proposition \ref{p:kFU-direct-sum}. \qed
\end{proof}

One of the most important classes of locally convex spaces is the class of free locally convex spaces. Following \cite{Mar}, the {\em  free locally convex space}  $L(X)$ on a Tychonoff space $X$ is a pair consisting of a locally convex space $L(X)$ and  a continuous map $i: X\to L(X)$  such that every  continuous map $f$ from $X$ to a locally convex space  $E$ gives rise to a unique continuous linear operator ${\bar f}: L(X) \to E$  with $f={\bar f} \circ i$. The free locally convex space $L(X)$ always exists and is essentially unique.

\begin{corollary} \label{t:kFU-free-LCS}
Let $X$ be a Tychonoff space. Then $L(X)$ is a $\kappa$-Fr\'{e}chet--Urysohn space if and only if $X$ is finite.
\end{corollary}

\begin{proof}
It is well known that $L(D)$ over a countably infinite discrete space $D$ is topologically isomorphic to $\phi$.
By Theorem 1.2 %\ref{t:Ascoli-L(X)}
of \cite{Gabr-L(X)-Ascoli}, $L(X)$ is an Ascoli space if and only if $X$ is a countable discrete space. This fact, Theorem \ref{t:k-FU-seq-Ascoli} and Proposition \ref{p:kFU-direct-sum} immediately imply   the assertion. \qed
\end{proof}

%%%%%%%%%%%%%%%%%%%%%%%%%%%
%%%%%%%%%%%%%%%%%%%%%%%%%%%
%%%%%%%%%%%%%%%%%%%%%%%%%%%
%%%%%%%%%%%%%%%%%%%%%%%%%%%

\bibliographystyle{amsplain}

\end{document}